\documentclass[a4paper,10pt]{article}
\usepackage{amstext,amsmath,amssymb,mathrsfs}
\usepackage{amsthm}
\usepackage{bm,color}
\usepackage[dvipdfmx]{graphicx}
\usepackage{here}
\usepackage{setspace}
\usepackage{cases}

\renewcommand\hat{\widehat}

\def\e{\epsilon}

\def\P{\mathbb{P}}

\def\1{\bf{1}}
\def\R{\mathbb{R}}

\def\Z{\mathbb{Z}}

\def\1{{\bf 1}}

\newtheorem{theorem}{Theorem}[section]
\newtheorem{lemma}[theorem]{Lemma}
\newtheorem{proposition}[theorem]{Proposition}

\newtheorem{definition}[theorem]{Definition}
\newtheorem{ex}[theorem]{Example}
\newtheorem{assumption}[theorem]{Assumption}

\theoremstyle{definition}
\newtheorem{remark}[theorem]{Remark}

\makeatletter
 
 \@addtoreset{equation}{section}
\makeatother

\title{On the KPZ scaling and the KPZ fixed point for TASEP}
\author{Yuta Arai \thanks{Platform for Arts and Science, Chiba University of Commerce, Ichikawa-shi 263-8522, Japan. Email: yutaarai@cuc.ac.jp}}
\date{}

\begin{document}

\maketitle

\maketitle

\begin{abstract}
We consider all totally asymmetric simple exclusion processes (TASEPs) whose transition probabilities are given in the Sch$\ddot{\rm u}$tz-type formulas and which jump with homogeneous rates.
We show that the multi-point distribution of particle positions and the coefficient of KPZ scaling are described using the probability generating function of the distribution followed when the rightmost particle jumps.
For all TASEPs satisfying certain assumptions,  We also prove the pointwise convergence of the kernels appearing in the joint distribution of particle positions to those appearing in the KPZ fixed point formula.
Our result generalizes the result of Matetski, Quastel, and Remenik \cite{Quastel}.
\end{abstract}

\section{Introduction}
The KPZ universal class was introduced in \cite{Kardar} to describe the universality of the growth model of an interface.
The totally asymmetric simple exclusion process (TASEP) is one of the most typical interacting stochastic particle systems.
It can be interpreted as a stochastic interface growth model belonging to the KPZ universal class.
Furthermore, the TASEP is known as an important model for studying the KPZ universality because its distribution function can be calculated for some quantities.

Research on the KPZ universality of TASEP has been actively conducted since around 2000.
First, in the case of the step initial condition, by considering the relationship stochastically growing Young diagram and TASEP, Johansson  \cite{KJohansson} has derived the one-point limit distribution of the particle current by using the RSK correspondence.
In this case, the limiting distribution turned out to be the GUE Tracy-Widom distribution from random matrix theory \cite{TW94}.
As a related work, in the case of the flat polynuclear growth (PNG) model, the one-point limit distribution of the height distribution has been obtained in \cite{Prahofer}.
In addition, for the last passage percolation, similar results have been derived in \cite{BaikR, Baik}.
The results of \cite{BaikR, Baik, Prahofer} include the result of the one-point limit distribution of particle current for the periodic initial condition in the language of TASEP. It turned out that the limiting distribution is the GOE Tracy-Widom distribution from random matrix theory \cite{TW96}.

The above are the results for one-point fluctuations, but many results for multi-point fluctuations have also been given.
For the case corresponding to the step initial condition, the Fredholm determinant formula for the limiting multi-point distribution has been obtained in the PNG model with different settings \cite{Johansson, PrSpAiry2}.
In this case, the limiting process characterized by the multi-point distribution is the $\rm Airy_{2}$ process.
On the other hand, for the periodic initial condition, the Fredholm determinant formula for the limiting multi-point distribution has been derived in the continuous time TASEP \cite{Sasamoto, TSasamoto} by using the result of the transition probability in TASEP \cite{Schutz}.
In this case, the limiting process characterized by the multi-point distribution is called the $\rm Airy_{1}$ process.
The technique in \cite{Sasamoto, TSasamoto} has been applied to various models.
Therefore, the limit distribution of the multi-point distribution has been obtained for the TASEP and PNG models with different settings \cite{BoFe, Borodin, Ferrari}.

The case of generalized initial conditions for particle positions has also been studied.
Matetski, Quastel, and Remenik \cite{Quastel} first extended the method of \cite{Sasamoto, TSasamoto} to get the limit distribution of multi-point distribution in the continuous time TASEP for arbitrary initial conditions:
In \cite{Sasamoto, TSasamoto}, the correlation kernel for the Fredholm determinant was expressed in terms of the biorthogonal functions $\Psi^{n}_k(x)$ and $\Phi^{n}_k(x)$.
However, there was the problem that $\Phi^{n}_k(x)$ does not have an explicit representation while $\Psi^{n}_k(x)$ does.
Therefore, it was not clear how to take the KPZ scaling limit of this kernel.
Matetski, Quastel, and Remenik \cite{Quastel} solved this problem. 
They represent the function $\Phi^{n}_k(x)$ by the hitting probability of the geometric random walk.
From Donsker's invariance principle, the hitting time of the geometric random walk converges to the hitting time of the Brownian motion when the time-space limit is taken, so this representation of $\Phi^{n}_k(x)$ allows us to take the KPZ scaling limit. 
Based on this method, they have derived the limit distribution of multi-point distribution in the continuous time TASEP for arbitrary initial conditions.
The limiting process with this limit distribution of the multi-point distribution is known as the KPZ fixed point.
The KPZ fixed point has also been obtained in the one-sided reflected Brownian motion \cite{Mihai} and two variations of discrete time TASEP with geometric and Bernoulli jumps \cite{Yuta} by using the method of \cite{Quastel}.
There have also been various other interesting progresses on the KPZ fixed point for example in \cite{Corwin, Sarkar, Virag}.

There have been studies to get the distribution of particle positions in the discrete time TASEP using the result of \cite{DW08}: 
Dieker and Warren \cite{DW08} have derived the transitive kernels of the four processes that correspond to the four variants of the discrete time TASEP by using the RSK correspondence.
Matetski and Remenik \cite{Daniel} have given the distribution of particle positions in the four variants of the discrete time TASEP from the above result and the method of \cite{Quastel}.
They have also obtained a formula for the distribution of particle positions that can be applied for example to continuous time TASEP and discrete time TASEP with sequential update.
In \cite{Zygouras}, they have generalized the method of \cite{DW08} so that it can be applied to the discrete time Bernoulli TASEP with particle- and time-inhomogeneous rates.
Combining the above method with the method of \cite{Quastel}, they gave the distribution of particle positions in the discrete time Bernoulli TASEP with particle- and time-inhomogeneous rates.
However, in \cite{Zygouras,Daniel}, the formula for obtaining the KPZ fixed point that can be uniformly applied to TASEP with different settings was not derived.

In this paper, we consider all TASEPs whose transition probabilities are given in the Sch$\ddot{\rm u}$tz-type formulas and which jump with homogeneous rates.
Then we show that the distribution of particle positions can be described by using the probability generating function of the distribution followed when the rightmost particle jumps.
We remark that this method using the probability generating function is quite different from the method of \cite{Zygouras,Daniel}.
We also state that the coefficient of KPZ scaling used to get the KPZ fixed point can be expressed by the probability generating function of the distribution followed when the rightmost particle jumps. 
Furthermore, we show the property of the coefficient of KPZ scaling.
Finally, by generalizing the method of \cite{Quastel}, we prove the pointwise convergence of the kernels appearing in the joint distribution of particle positions to those appearing in the KPZ fixed point formula for all TASEPs that satisfy certain assumptions, for example, the continuous time TASEP, the discrete time Bernoulli TASEP with sequential update, and the discrete time geometric TASEP with parallel update.
It implies that our method can adapt to multiple models, not to only one model.
Note that our method derives the KPZ fixed point even in the case of the continuous time TASEP with jump rate $\beta\in (0,\infty)$ where the KPZ fixed point has not been given (see Example \ref{ex} and Appendix \ref{Appendix}).

The paper is organized as follows:
In Section \ref{sec2}, we state the TASEPs whose transition probabilities are given in the Sch$\ddot{\rm u}$tz-type formulas (see Assumption \ref{as1}). 
We also give our main result: the Fredholm determinant formula for the TASEPs satisfying Assumption \ref{as1} (Theorem \ref{Main}), the property of the coefficient of KPZ scaling (Theorem \ref{main2}), and the KPZ scaling limit in the TASEPs satisfying Assumption \ref{as1} when Assumption \ref{imas}, and Assumption \ref{limcon} hold (Theorem \ref{special}).
In Section \ref{sec3}, for the TASEPs which satisfy Assumption 2.1, we show Theorem \ref{Main} after the transition probabilities are represented by the probability generating function of the distribution followed when the rightmost particle jumps.
In Section \ref{sec4}, we prove Theorem \ref{main2}.
In Section \ref{sec5}, we give proofs of Theorem \ref{special} and Proposition \ref{scaling}.
In Appendix \ref{Appendix}, we use our method to show that the KPZ fixed point is obtained in the continuous time TASEP with jump rate $\beta\in(0, \infty)$.
The key to our proofs is the saddle point analysis for the kernels by using the probability generating function of the distribution followed when the rightmost particle jumps. 
\section{Models and results}
\label{sec2}
\subsection{Models}
We consider the TASEPs on $\mathbb{Z}$.
Each particle independently and stochastically jumps to the right only if the target site is empty and cannot move if the target site is occupied by other particles.
The above represents the exclusion rule.

We mainly focus on the position of each particle.
We put $t\in\mathbb{Z}$ or $t\in\mathbb{R}$ according to the version.
Then we define  $X_t(i)\in\mathbb{Z}$ as a position of the $i$th particle at time $t$.
The dynamics of the TASEPs preserve the order of the particles, that is, 
\begin{equation*}
\cdots<X_{t}(i+2)<X_{t}(i+1)<X_{t}(i)<X_{t}(i-1)<X_{t}(i-2)<\cdots
\end{equation*}
where the particles at $\pm\infty$ are playing no role in the dynamics when adding $\pm\infty$ into the state space.

 Now we set
\begin{equation*}
\Omega_N=\{\vec{x}=(x_N,x_{N-1},\cdots,x_1)\in\Z^N: x_N<\dots<x_2<x_1\}
\end{equation*}
as the Weyl chamber, whose elements express the particle positions of the TASEPs.
Also, we put 
\begin{equation}
\label{s21}
F_n(x,t)=\frac{(-1)^n}{2\pi i}\oint_{\Gamma_{0,1}}dw\frac{(1-w)^{-n}}{w^{x-n+1}}\mathcal{M}(t, w)
\end{equation}
where $\Gamma_{0,1}$ is any simple loop oriented anticlockwise which includes $w=0$ and $w=1$, and $\mathcal{M}(t, w)$ is analytic on $\{w\in\mathbb{C}: |w|<R\}$ with the radius $R\geq1$.

In this paper, we deal with the TASEP which satisfies the following assumption, where we do not consider the TASEP whose jump rate or jump probability changes depending on time. \begin{assumption}
\label{as1}
The transition probability from $\vec{y}\in\Omega_N$ to $\vec{x}\in\Omega_N$ is given by 
\begin{equation}
\label{s22}
\mathbb{P}(X_t=\vec{x}|X_0=\vec{y})=\det[F_{i-j}(x_{N+1-i}-y_{N+1-j}, t)]_{1\leq i,j\leq N},
\end{equation}
where $X_t=(X_t(1), X_t(2),\dots, X_t(N))$ are the locations of a system of particles. 
\end{assumption}
Assumption \ref{as1} is fulfilled in many TASEPs illustrated in the following example.
\begin{ex}
\label{ex}
{\rm Here we introduce three typical examples.}
\begin{itemize}
\item{\rm The continuous time TASEP with jump rate $\beta$}\\
The continuous time TASEP was introduced in {\rm \cite{Spitzer}} as a mathematical model.
The process $X_t$, $t\in\mathbb{R}_{\geq 0}$ evolves as follows: 
each particle independently attempts to jump to the right neighboring site at rate $\beta\in (0, \infty)$ provided this site is empty. 
The continuous time TASEP is a Markov process with the generator $L$ defined  as follows:
We put $\eta=\{\eta(x): x\in\mathbb{Z}\}\in\{0,1\}^{\mathbb{Z}}$ as  a particle configuration where $\eta(x)=1$ means the site $x$ is occupied by a particle while $\eta(x)=0$ means it is empty.
Then the generator $L$ acting on cylinder functions $f:\{0,1\}^{\mathbb{Z}}\rightarrow\mathbb{R}$ is introduced by
\begin{equation*}
(Lf)(\eta)=\beta\sum_{x\in\mathbb{Z}}\eta(x)(1-\eta(x+1))(f(\eta^{x,x+1})-f(\eta))
\end{equation*}
where
\begin{equation*}
\eta(x)=
\begin{cases}
1,&\text{if the site is occupied by a particle,}\\
0,&\text{if the site $x$ is empty,}
\end{cases}
\end{equation*}
and
$\eta^{x,x+1}$ is the configuration $\eta$ with the occupations at site $x$ and $x+1$ have been interchanged, that is,
\begin{equation*}
\eta^{x,x+1}(y)=
\begin{cases}
\eta(x+1)& \text{for~} y=x,\\
\eta(x)& \text{for~} y=x+1,\\
\eta(y)& \text{otherwise}.
\end{cases}
\end{equation*}
The transition probability of $X_t$ is given by {\rm \cite{Schutz}} using Bethe ansatz:
\begin{equation*}
\mathbb{P}(X_t=\vec{x}|X_0=\vec{y})=\det[F_{i-j}(x_{N+1-i}-y_{N+1-j}, t)]_{1\leq i,j\leq N},
\end{equation*}
with
\begin{equation*}
F_n(x,t)=\frac{(-1)^n}{2\pi i}\oint_{\Gamma_{0,1}}dw\frac{(1-w)^{-n}}{w^{x-n+1}}e^{\beta t(w-1)}
\end{equation*}
where $\Gamma_{0,1}$ is any simple loop oriented anticlockwise which includes $w=0$ and $w=1$.
It is clear that this model satisfies Assumption \ref{as1} with the function
\begin{equation}
\label{beta232}
\mathcal{M}(t, w)=e^{\beta t(w-1)}.
\end{equation}
Note that when $\beta\in (0,1)$, this model can be interpreted as the continuous time version of the discrete time Bernoulli TASEP introduced next.
The KPZ fixed point has been derived in {\rm \cite{Quastel}} when $\beta=1$, but our results show that the KPZ fixed point can also be obtained when $\beta\in (0, \infty)$ (see Appendix \ref{Appendix}).
\item{\rm The discrete time Bernoulli TASEP with sequential update}\\
The discrete time Bernoulli TASEP with sequential update on $\mathbb{Z}$ was studied previously in {\rm \cite{aniso}} as a marginal of dynamics on Gelfand-Tsetlin patterns which preserve the class of Schur processes.
The evolution of the process $X_t$, $t\in\mathbb{Z}_{\geq 0}$  is given by the recursion relation
\begin{equation*}
X_{t+1}(1)=X_{t}(1)+w_{t+1, 1}
\end{equation*}
and
\begin{equation*}
X_{t+1}(i)=\min\left\{X_{t}(i)+w_{t+1, i}, X_{t+1}(i-1)-1\right\},~~~~ i=2, 3, \dots, N
\end{equation*}
where $w_{t, i}$ are independent random variables following the Bernoulli distribution with parameter $p\in(0, 1)$.
The transition probability of this model is given by {\rm \cite{Rakos}}:
\begin{equation*}
\mathbb{P}(X_t=\vec{x}|X_0=\vec{y})=\det[F_{i-j}(x_{N+1-i}-y_{N+1-j}, t)]_{1\leq i,j\leq N},
\end{equation*}
with
\begin{equation*}
F_n(x,t)=\frac{(-1)^n}{2\pi i}\oint_{\Gamma_{0,1}}dw\frac{(1-w)^{-n}}{w^{x-n+1}}(1+p(w-1))^t
\end{equation*}
where $\Gamma_{0,1}$ is any simple loop oriented anticlockwise which includes $w=0$ and $w=1$.
One can readily see that this model satisfies Assumption \ref{as1} with the function
\begin{equation*}
\mathcal{M}(t, w)=(1+p(w-1))^t.
\end{equation*}
\item{\rm The discrete time geometric TASEP with parallel update}\\
The discrete time geometric TASEP with parallel update on $\mathbb{Z}$ was studied previously in {\rm \cite{Jon}} as a marginal of dynamics on Gelfand-Tsetlin patterns which preserve the class of Schur processes.
The evolution of the process $X_t$, $t\in\mathbb{Z}_{\geq 0}$  is given by the recursion relation
\begin{equation*}
X_{t+1}(1)=X_{t}(1)+\hat{w}_{t+1, 1}
\end{equation*}
and
\begin{equation*}
X_{t+1}(i)=\min\left\{X_{t}(i)+\hat{w}_{t+1, i}, X_{t+1}(i-1)-1\right\},~~~~ i=2, 3, \dots, N
\end{equation*}
where $\hat{w}_{t, i}$ are independent random variables following the Geometric distribution with parameter $\alpha\in(0, 1)$.
The transition probability of this process is given by {\rm \cite{Yuta, DW08, Daniel}}:
\begin{equation}
\label{geogeo1}
\mathbb{P}(X_t=\vec{x}|X_0=\vec{y})=\det[F_{i-j}(x_{N+1-i}-y_{N+1-j}, t)]_{1\leq i,j\leq N},
\end{equation}
with
\begin{equation}
\label{geogeo2}
F_n(x,t)=\frac{(-1)^n}{2\pi i}\oint_{\Gamma_{0,1}}dw\frac{(1-w)^{-n}}{w^{x-n+1}}\left(\frac{1-\alpha}{1-\alpha w}\right)^t
\end{equation}
where $\Gamma_{0,1}$ is any simple loop oriented anticlockwise which includes $w=0$ and $w=1$.
Note that it was first shown in {\rm \cite{DW08}} that the transition probabilities are given by determinants:
Dieker and Warren {\rm \cite{DW08}} have represented the transition probabilities by using certain sums involving symmetric polynomials.  
On the other hand, the expression of the transition probability by contour integral formulas like \eqref{geogeo2} has first been given in {\rm \cite{Yuta}}.
Besides, it was shown in {\rm \cite{Daniel}} that the expression of transition probability in {\rm \cite{Yuta}} and the expression of transition probability in {\rm \cite{DW08}} are equivalent.
It is easy to see that this model satisfies Assumption \ref{as1} with the function
\begin{equation*}
\mathcal{M}(t, w)=\left(\frac{1-\alpha}{1-\alpha w}\right)^t.
\end{equation*}
\end{itemize}
\end{ex}
\subsection{Results}
In this subsection, we state our main results.
\subsubsection{The representation of the distribution of the particle positions}
Now we give a single Fredholm determinant formula for the joint distribution
of the particle position in TASEP satisfies Assumption \ref{as1}.
For describing our results, we state some definitions.
\begin{definition}[epigraph and hypograph]
\label{epihy}
For a real single-valued function $\hat{f}:\mathbb{A}\to(-\infty, \infty]$
with (in general an uncountable) domain $\mathbb{A}$, we set
\begin{align*}
{\rm epi}(\hat{f})=\{(x,y) : y\geq \hat{f}(x)\},~{\rm hypo} (\hat{f})=\{(x,y) : y\leq \hat{f}(x)\}.
\end{align*}
\end{definition}

\begin{definition}
\label{defRWtau}
We put ${RW}_m,~m=0,1,2\dots$ as the position of a random walker with Geom$[\frac{1}{2}]$ jumps strictly to the left starting at some fixed site $c$, 
that is to say,
\begin{align*}
{RW}_m=c-\chi_1-\chi_2-\cdots -\chi_m,
\end{align*}
where $\chi_i,~i=1,2,\dots$ are the i.i.d. random variable with
$
\P(\chi_i=k)=1/2^{k+1},~k=0,1,2,\dots.
$

We also set the stopping time
\begin{equation}
\label{Tau}
\tau=\min\{m\geq 0 : RW_m>X_0(m+1)\}
\end{equation}
where $\tau$ is the hitting time of the strict epigraph of the curve $(X_0(k+1))_{k=0,\dots, n-1}$ by the random walk $RW_k$, 
$X_0(m)$ is constant and defined only $m\le N$ when the number of particles is $N$.
\end{definition}

At last we set the multiplication operators.
\begin{definition}
\label{chi}
For a fixed vector $a\in\mathbb{R}^m$ and indices $n_1<\cdots<n_m$, we define
\begin{align*}
\chi_{a}(n_j,x)=\1_{x>a_j},~~\bar{\chi}_{a}(n_j,x)=\1_{x\leq a_j}.
\end{align*}
 as the multiplication operators acting on the space $\ell^2(\{n_1,\dots,n_m\}\times\mathbb{Z})$(or acting on the space $L^2(\{x_1,\dots,x_m\}\times\mathbb{R})$).
 \end{definition}
When considering the distribution of particle positions, we assume that the rightmost particle exists and is labeled $1$.
Now we remark that the following:
The rightmost particle of TASEP $X_t(1)$ is a (right) one-sided jump random walk or a compound Poisson process because the exclusion rule does not work.
Therefore 
\begin{itemize}
\item If $t\in\mathbb{Z}_{\geq 0}$, then
\begin{equation}
\label{m11}
X_t(1):=Y_1+Y_2+\dots+Y_t
\end{equation}
where $Y_1, Y_2, \dots, Y_t$ are independent and identically distributed non-negative integer-valued random variables.
\item If $t\in\mathbb{R}_{\geq 0}$, then
\begin{equation}
\label{m12}
X_t(1)=S_{N_t}:=Z_1+Z_2+\dots+Z_{N_t}
\end{equation}
where $Z_1, Z_2, \dots$ are independent and identically distributed non-negative integer-valued random variables, $N_t$ is Poisson process with parameter $\lambda\in(0, \infty)$, independent of the process $S_n$, $n\in\mathbb{Z}_{\geq 0}$.
\end{itemize}

Noting that \eqref{m11} and \eqref{m12}, we have the following result.
\begin{proposition}
\label{m13}
We consider the TASEP that satisfies Assumption \ref{as1}.
Then
\begin{equation}
\mathcal{M}(t, w)=\mathcal{M}(w)^t
\end{equation}
where 
\begin{equation}
\label{pgf}
\mathcal{M}(w)=
\begin{cases}
G_{Y_1}(w) & \text{if $t\in\mathbb{Z}_{\geq 0}$,}\\
G_{X}(G_{Z_1}(w)) & \text{if $t\in\mathbb{R}_{\geq 0}$,}
\end{cases}
\end{equation}
$G_{Z}(w)$ is a probability generating function of the non-negative integer-valued random variable $Z$, that is,
\begin{equation}
\label{jyuuyou5}
G_{Z}(w)=\sum_{k=0}^{\infty}w^k\mathbb{P}(Z=k),
\end{equation}
$Y_1$ is defined in \eqref{m11},  $Z_1$ is defined in \eqref{m12} and $X$ is Poisson random variable with parameter $\lambda\in(0, \infty)$.
\end{proposition}
This proof is given in Section \ref{s31}.

Furthermore, we get the following by using Proposition \ref{m13}.
\begin{theorem}
\label{s88}
We consider the TASEP that satisfies Assumption \ref{as1}.
Then the transition probability of TASEP is given as the following:
\begin{equation*}
\mathbb{P}(X_t=\vec{x}|X_0=\vec{y})=\det[\overline{F}_{i-j}(x_{N+1-i}-y_{N+1-j}, t)]_{1\leq i,j\leq N}
\end{equation*}
where $\vec{x},\vec{y}\in\Omega_N$, 
\begin{equation}
\label{s81}
\overline{F}_n(x,t)=\frac{(-1)^n}{2\pi i}\oint_{\Gamma_{0,1}}dw\frac{(1-w)^{-n}}{w^{x-n+1}}\mathcal{M}(w)^t,
\end{equation}
where $\Gamma_{0,1}$ is any simple loop oriented anticlockwise which includes $w=0$ and $w=1$ and $\mathcal{M}(w)$ is defined in \eqref{pgf}.
\end{theorem}
Proof is given in Section \ref{s31}.

From Theorem \ref{s88} and Theorem 1.2 of \cite{Daniel}, we obtain the following results.
The following results represent that the distribution of particle positions can be expressed by the probability generating function of the distribution followed when the rightmost particle jumps.
\begin{theorem}
\label{Main}
We consider the TASEP which satisfies Assumption \ref{as1}. 
Let $t\in\mathbb{Z}$ or $t\in\mathbb{R}$. 
Also, we put $X_t(j),~j\in\Z$ as a position of the $j$th particle at time $t$.
Assume that the initial positions $X_0(j)\in\Z$ for $j=1,2,\dots$ are arbitrary constants satisfying $X_0(1)>X_0(2)>\cdots$
while $X_0(j)=\infty$ for $j\leq 0$.

For $n_j\in\Z_{\ge 1}~j=1,2,\dots,M$ with $1\leq n_1<n_2<\cdots<n_M$, and ${a}=(a_1,a_2, \dots, a_M)\in\Z^M$, we get
\begin{equation}
\label{pro}
\mathbb{P}(X_t(n_j)>a_j, j=1,\dots,M)=\det(I-\bar{\chi}_{a}K_t\bar{\chi}_{a})_{\ell^2(\{n_1,\dots,n_M\}\times\mathbb{Z})}.
\end{equation}
Here $\bar{\chi}_{\bm a}(n_j,x)$ is introduced in Definition \ref{chi} and 
\begin{align}
\label{Kt}
&K_t(n_i, x  ; n_j, y)=-Q^{n_j-n_i}(x,y)\1_{n_i<n_j}+(\mathfrak{S}_{-t, -n_i})^{*}\bar{\mathfrak{S}}^{{\rm epi} (X_0)}_{-t,n_j}(x,y),
\\
&Q^m(x,y)=\frac{1}{2^{x-y}}\binom{x-y-1}{m-1}\1_{x\geq y+m},
\\
\label{stn}
&\mathfrak{S}_{-t,-n}(z_1,z_2)=\frac{1}{2\pi i}\oint_{\Gamma_0}dw\frac{(1-w)^n}{2^{z_2-z_1}w^{n+1+z_2-z_1}}\left\{\frac{\mathcal{M}(w)}{\mathcal{M}\left(\frac{1}{2}\right)}\right\}^{t},
\\
\label{bstn}
&\bar{\mathfrak{S}}_{-t,n}(z_1,z_2)=\frac{1}{2\pi i}\oint_{\Gamma_0}dw\frac{(1-w)^{z_2-z_1+n-1}}{2^{z_1-z_2}w^{n}}\left\{\frac{\mathcal{M}(1-w)}{\mathcal{M}\left(\frac{1}{2}\right)}\right\}^{-t},
\\
\label{sepi}
&\displaystyle\bar{\mathfrak{S}}^{\rm epi(X_0)}_{-t,n}(z_1, z_2)=\mathbb{E}_{RW_0=z_1}\left[\bar{\mathfrak{S}}_{-t, n-\tau}(RW_{\tau}, z_2)\1_{\tau<n}\right],
\end{align}
where
$\Gamma_0$ is a simple counterclockwise loop around $0$ not enclosing any other poles and $\mathcal{M}(w)$ is defined by \eqref{pgf}.
The superscript $\rm epi(X_0)$ in~\eqref{sepi} refers to the fact that $\tau$ is the hitting time of the strict epigraph of the curve $(X_0(k+1))_{k=0,\dots, n-1}$ by the random walk $RW_k$ defined in Definition \ref{defRWtau}.
\end{theorem}
This proof is given in Section \ref{s32}.
\begin{remark}
Theorem \ref{s88} implies that applying Proposition \ref{m13} to Assumption \ref{as1} leads to Assumption 1.1 of \cite{Daniel}.
Therefore, in Theorem \ref{Main}, the distribution of particle positions has been obtained under Assumption \ref{as1} which is a weaker assumption than Assumption 1.1 in \cite{Daniel}.
Moreover, Theorem \ref{Main} shows that the function $\mathcal{M}(w)$, which was not obtained in the explicit form in Theorem 1.2 of \cite{Daniel}, is given by the probability generating function of the distribution followed when the rightmost particle jumps.
\end{remark}
\subsubsection{The coefficient of KPZ scaling and the KPZ scaling limit }
Now we introduce our result on the scaling limit of the joint distribution function in Theorem~\ref{Main}. 
Our main results include the result of continuous-time TASEP in \cite{Quastel} and the results of discrete-time TASEP in \cite{Yuta} and can adapt to multiple models, not to only one model.

To see the universal behavior of the fluctuations, we focus on
the height function.
\begin{definition}
\label{relation1}
For $z\in\mathbb{Z}$, we give the TASEP height function related to $X_t$ by
\begin{equation}
\label{height}
\displaystyle h_t(z)=-2(X^{-1}_t(z-1)-X^{-1}_0(-1))-z
\end{equation}
where 
\begin{equation}
X^{-1}_t(u)=\min\{k\in\mathbb{Z}:X_t(k)\leq u\}
\end{equation}
is the label of the rightmost particle which sits to the left of, or at, $u$ at time $t$
and we fix $h_0(0)=0$.
\end{definition}

Note that we can represent it by
\begin{equation}
h_t(z+1)=h_t(z)+\hat{\eta}_t(z)
\end{equation}
where
\begin{equation*}
\hat{\eta}_t(z)
=\begin{cases}
1&\text{if there is a particle at $z$ at time $t$,}\\
-1&\text{if there is no particle at $z$ at time $t$.}
\end{cases}
\end{equation*}
We can also extend the height function to a continuous function of $x\in\mathbb{R}$ by linearly interpolating between the integer points.

Because the TASEP is known to belong to the Kardar-Parisi-Zhang (KPZ) universality class, we expect that the proper scaling of the height function is
\begin{align}
\frac{h_t(x)-At}{Ct^{\frac13}}, \text{~with~} x=Bt^{2/3}.
\label{kpzscaling1}
\end{align}
The above means that the height average of the TASEP grows as $t^1$ with speed $A$ which is a constant.
In addition, the fluctuation of the height around the average is order $t^{1/3}$ against the $t^{1/2}$ of the scaling in the central limit theorem. 
The scaling exponent of the $x$-direction is $2/3$ which is twice of the $h$-direction $1/3$.
This suggests that the path of the height function becomes the Brownian motion like. 
It is well known that the exponents $(1/3, 2/3)$ are universal and characterize the KPZ universality class.
In previous studies, the constants $A$, $B$, and $C$ are obtained only for each model.

In this paper, we introduce that the constants $A$, $B$, and $C$ can be written without depending on the model when a specific assumption is imposed.
First, we put 
\begin{equation}
\label{gamgam}
\gamma(w)=\frac{\mathcal{M}\left({\textstyle\frac{1}{2}(1-w)}\right)}{\mathcal{M}\left({\textstyle\frac{1}{2}}\right)}
\end{equation}
where $\mathcal{M}(w)$ is defined in \eqref{pgf}.
Next, we introduce the following assumption.
\begin{assumption}
\label{imas}
\begin{equation}
\label{ftusi}
\gamma^{(3)}(0)-3\gamma^{(2)}(0)\gamma^{'}(0)+2\{\gamma^{'}(0)\}^3-2\gamma^{'}(0)>0
\end{equation}
where $\gamma^{(n)}(w)$ is the $n$-th derivative of $\gamma(w)$.
\end{assumption}
\begin{remark}
Several well-known TASEP models meet the above assumption.
For example, it is easy to see that the continuous time TASEP, the discrete time Bernoulli TASEP with sequential update, and the discrete time geometric TASEP with parallel update satisfy Assumption \ref{imas}.
However, we can give the probability generating function $\mathcal{M}(w)$ \eqref{pgf} that does not satisfy Assumption \ref{imas}.
For example, we consider the case where the update rule is given as follows:
\begin{equation*}
\mathbb{P}(X_{t+1}(1)=a_1+b|X_{t}(1)=a_1)\\
=
\begin{cases}
p &\text{for $b=4$,}
\\
1-p & \text{for $b=0$,}
\\
0 & \text{otherwise,}
\end{cases}
\end{equation*}
where $a_1\in\mathbb{Z}$ and $0<p<1$.
Then 
\begin{equation*}
\mathcal{M}(w)=1-p+pw^4
\end{equation*}
and
\begin{equation*}
\gamma^{(3)}(0)-3\gamma^{(2)}(0)\gamma^{'}(0)+2\{\gamma^{'}(0)\}^3-2\gamma^{'}(0)=\frac{p(1-p)(23p-16)}{16\left\{\mathcal{M}(\frac{1}{2})\right\}^3}.
\end{equation*}
Therefore, when $0<p<\frac{16}{23}$, 
\begin{equation*}
\gamma^{(3)}(0)-3\gamma^{(2)}(0)\gamma^{'}(0)+2\{\gamma^{'}(0)\}^3-2\gamma^{'}(0)<0.
\end{equation*}
From the above, we see that Assumption \ref{imas} is necessary.
\end{remark}
Supposing that Assumption \ref{imas}, the constants $A$, $B$, and $C$ are given as follows:
\begin{align}
A= \frac{2\{\gamma^{'}(0)+\{\gamma^{'}(0)\}^2-\gamma^{(2)}(0)\}}{\gamma^{(3)}(0)-3\gamma^{(2)}(0)\gamma^{'}(0)+2\{\gamma^{'}(0)\}^3-2\gamma^{'}(0)},B=2, C=1.
\label{kpzscaling2}
\end{align}
By the property of the height function, we put the scaled height, which is equivalent to~\eqref{kpzscaling2}.
This is known as \lq{}\lq{}1:2:3 scaling\rq\rq{} which is defined in~\cite{Quastel}.
\begin{definition}
For $\mathbf{t}\in\R_{\ge 0}$ and $\mathbf{x}\in\R$, we set the scaling height function
\begin{equation}
\label{Hei}
\displaystyle\hat{h}^{\varepsilon}(\mathbf{t},\mathbf{x})=\varepsilon^{\frac{1}{2}}\left[h_{t}(x)+\frac{2\{\gamma^{(2)}(0)-\{\gamma^{'}(0)\}^2-\gamma^{'}(0)\}}{\gamma^{(3)}(0)-3\gamma^{(2)}(0)\gamma^{'}(0)+2\{\gamma^{'}(0)\}^3-2\gamma^{'}(0)}\varepsilon^{-\frac{3}{2}}\mathbf{t}\right],
\end{equation}
where $t$ and $x$ are scaled as
\begin{align}
\label{txscaling}
t=\frac{2}{\gamma^{(3)}(0)-3\gamma^{(2)}(0)\gamma^{'}(0)+2\{\gamma^{'}(0)\}^3-2\gamma^{'}(0)}\varepsilon^{-\frac{3}{2}}\mathbf{t},
~
x=2\varepsilon^{-1}\mathbf{x}.
\end{align}
\end{definition}
\begin{remark}
\eqref{Hei} and \eqref{txscaling} show that the coefficient of $\varepsilon^{-\frac{3}{2}}\mathbf{t}$ can be expressed by the probability generating function of the distribution followed when the rightmost particle jumps.
This implies that the coefficient of $\varepsilon^{-\frac{3}{2}}\mathbf{t}$ is described by the probability generating functions of the particles unaffected by the exclusion rule.
\end{remark}
In \eqref{Hei} and \eqref{txscaling}, we focus on the denominator and numerator of the coefficient of $\varepsilon^{-\frac{3}{2}}\mathbf{t}$.
Then we can give the properties needed to obtain the scaling limit of distribution of height function.
\begin{theorem}
\label{main2}
We have
\begin{equation}
\label{pgfg2}
0\leq\gamma^{(2)}(0)-\{\gamma^{'}(0)\}^2-\gamma^{'}(0)<\infty,
\end{equation}
\begin{equation}
\label{pgfg3}
|\gamma^{(3)}(0)-3\gamma^{(2)}(0)\gamma^{'}(0)+2\{\gamma^{'}(0)\}^3-2\gamma^{'}(0)|<\infty
\end{equation}
where $\gamma(w)$ is introduced in \eqref{gamgam}.
\end{theorem}
This proof is given in Section \ref{sec4}.

Our goal is to compute the $\varepsilon\rightarrow 0$ limit of the joint distribution function
\begin{align}
\lim_{\varepsilon\xrightarrow{}0}\mathbb{P}_{\hat{h}^{\varepsilon}_0}(\hat{h}^{\varepsilon}(\mathbf{t}, \mathbf{x}_1)\leq\mathbf{a}_1,\dots, \hat{h}^{\varepsilon}(\mathbf{t}, \mathbf{x}_m)\leq\mathbf{a}_m)
\label{ldisfunc}
\end{align}
for $\mathbf{x}_1<\mathbf{x}_2< \dots <\mathbf{x}_m \in\mathbb{R}$ and $\mathbf{a}_1,\dots, \mathbf{a}_m \in\mathbb{R}$ where $\P_{\hat{h}^{\varepsilon}_0}(\cdot)$ is the probability measure and $\hat{h}^{\varepsilon} (0,x)$ is the initial height profile.
Here we introduce the assumption necessary to calculate \eqref{ldisfunc}.
\begin{assumption}
\label{limcon}
For $\theta\in [-\pi, -\frac{\pi}{3})\cup(\frac{\pi}{3}, \pi]$, 
\begin{equation}
\label{condi10}
E\log|2-e^{i\theta}|+D\log|\gamma(1-e^{i\theta})|<0
\end{equation}
and
\begin{equation}
\label{condi11}
F\log|2-e^{i\theta}|-D\log|\gamma(e^{i\theta}-1)|<0
\end{equation}
where 
\begin{equation}
\label{D}
D=\frac{2}{\gamma^{(3)}(0)-3\gamma^{(2)}(0)\gamma^{'}(0)+2\{\gamma^{'}(0)\}^3-2\gamma^{'}(0)},
\end{equation}
\begin{equation}
\label{E}
E=\frac{\gamma^{(2)}(0)-\{\gamma^{'}(0)\}^2-\gamma^{'}(0)}{\gamma^{(3)}(0)-3\gamma^{(2)}(0)\gamma^{'}(0)+2\{\gamma^{'}(0)\}^3-2\gamma^{'}(0)},
\end{equation}
\begin{equation}
\label{F}
F=\frac{\{\gamma^{'}(0)\}^2-\gamma^{(2)}(0)-\gamma^{'}(0)}{\gamma^{(3)}(0)-3\gamma^{(2)}(0)\gamma^{'}(0)+2\{\gamma^{'}(0)\}^3-2\gamma^{'}(0)}
\end{equation}
and $\gamma(w)$ is introduced in \eqref{gamgam}.
\end{assumption}
Note that considering the well-known TASEP, Assumption \ref{limcon} holds similarly to Assumption \ref{imas}.

We will prove that the limit \eqref{ldisfunc} converges to the joint distribution function characterizing the KPZ fixed point defined in~\cite{Quastel}.
Now we introduce the KPZ fixed point. First we define UC and LC.
\begin{definition}{\rm(UC} and {\rm LC}~{\rm \cite{Quastel})}.\\
We set {\rm UC} as the space of upper semicontinuous functions $\hat{h}:\mathbb{R}\to[-\infty, \infty)$ with $\hat{h}(x)\leq C_1+C_2|x|$ for some $C_1, C_2<\infty$ and $\hat{h}(x)>-\infty$ for some $x$
and {\rm LC} as ${\rm LC}=\{\hat{g} : -\hat{g}\in {\rm UC}\}$.
\end{definition}
Next we put the integral representation for the Airy function.
\begin{definition}
\label{airy}
the integral representation for the Airy function is given by
\begin{equation*}
\displaystyle{\rm Ai}(z)=\frac{1}{2\pi i}\int_{\langle}dw \ e^{\frac{1}{3}w^3-zw},
\end{equation*}
where $\langle$ is the positively oriented contour going the straight lines from $e^{-\frac{i\pi}{3}}\infty$ to $e^{\frac{i\pi}{3}}\infty$ through $0$.
\end{definition}
Now we are ready to state the KPZ fixed point (for more detail, see~\cite{Quastel}).
\begin{definition}[The KPZ fixed point~\cite{Quastel}]
The KPZ fixed point is the unique Markov process on {\rm UC}, $(\hat{h}(\mathbf{t},\cdot))_{\mathbf{t}>0}$
with transition probabilities
\begin{equation}
\label{kpzfpdet}
\displaystyle\mathbb{P}_{\hat{h}_0}(\hat{h}(\mathbf{t},\mathbf{x}_1)\leq \mathbf{a}_1,\dots, \hat{h}(\mathbf{t},\mathbf{x}_m)\leq \mathbf{a}_m)=\det\left(\mathbf{I}-\chi_{\mathbf{a}}\mathbf{K}^{{\rm hypo}(\hat{h}_0)}_{\mathbf{t},{\rm ext}}\chi_{\mathbf{a}}\right)_{L^2(\{\mathbf{x}_1,\dots, \mathbf{x}_m\}\times\mathbb{R})}
\end{equation}
where in LHS, $\mathbf{x}_1<\mathbf{x}_2< \dots <\mathbf{x}_m \in\mathbb{R}$ and $\mathbf{a}_1,\dots, \mathbf{a}_m \in\mathbb{R}$, $\hat{h}_0\in{\rm UC}$ and $\mathbb{P}_{\hat{h}_0}$ means the measure on the process with initial data $\hat{h}_0$. In RHS, we give the kernel by
\begin{align}
\label{funct}
&~\mathbf{K}^{{\rm hypo}(\hat{h}_0)}_{\mathbf{t},{\rm ext}}(\mathbf{x}_i, v ;\mathbf{x}_j, u)
\notag
\\
&=-\frac{1}{\sqrt{4\pi(x_j-x_i)}}\exp\left(-\frac{(u-v)^2}{4(x_j-x_i)}\right)\1_{\mathbf{x}_i<\mathbf{x}_j}+\left(\mathbf{S}^{{\rm hypo}(\hat{h}^{-}_0)}_{\mathbf{t},-\mathbf{x}_i}\right)^{*}\mathbf{S}_{\mathbf{t}, \mathbf{x}_j}(v, u),
\\
&
\label{function}
~\displaystyle\mathbf{S}_{\mathbf{t}, \mathbf{x}}(v, u)
=\mathbf{t}^{-\frac{1}{3}}e^{\frac{2\mathbf{x}^3}{3\mathbf{t}^2}-\frac{(v-u)\mathbf{x}}{\mathbf{t}}} {\rm Ai}(-\mathbf{t}^{-\frac{1}{3}}(v-u)+\mathbf{t}^{-\frac{4}{3}}\mathbf{x}^2),
\\
\label{functione}
&
~\displaystyle \mathbf{S}^{{\rm hypo}(\hat{h})}_{\mathbf{t},\mathbf{x}}(v,u)=\mathbb{E}_{B(0)=v}[\mathbf{S}_{\mathbf{t}, \mathbf{x}-\bm{\tau}'}(B(\bm{\tau}'),u)\1_{\bm{\tau}'<\infty}],
\end{align}
where $(A)^*$ represents the adjoint of an integral operator $A$, ${\rm Ai}(z)$ is defined in Definition \ref{airy} and $B(x)$ is a Brownian motion 
with diffusion coefficient $2$ and $\bm{\tau}'$ is the hitting time of ${\rm hypo}(\hat{h})$ introduced in Definition \ref{epihy}.
\end{definition}
\begin{remark}
By using he integral representation for the Airy function defined in Definition \ref{airy}, we find that $\mathbf{S}_{\mathbf{t}, \mathbf{x}}(v, u)$~\eqref{function} can be expressed as
\begin{align}
\displaystyle\mathbf{S}_{\mathbf{t}, \mathbf{x}}(v, u)=\frac{1}{2\pi i}\int_{\langle}dw \ e^{\frac{\mathbf{t}}{3}w^3+\mathbf{x}w^2+(v-u)w}.
\end{align}
\end{remark}
Now we suppose that the limit
\begin{equation}
\label{Height}
\hat{h}_0=\lim_{\varepsilon\xrightarrow{}0}\hat{h}^{\varepsilon}(0, \cdot)
\end{equation}
exists. By(\ref{height}) and (\ref{Height}), (\ref{initial}), this assumption is rewritten as 
\begin{equation}
\label{varx}
\displaystyle\varepsilon^{\frac{1}{2}}[(X^{\varepsilon}_0)^{-1}(\mathbf{x})+2\varepsilon^{-1}\mathbf{x}-2]\xrightarrow[\varepsilon\xrightarrow{}0]{}-\hat{h}_0(-\mathbf{x}),
\end{equation}
where $(X^{\varepsilon}_0)^{-1}(\mathbf{x}):=2X^{-1}_0(-2\varepsilon^{-1}\mathbf{x}-1)$, LHS of \eqref{varx} is interpreted as a linear interpolation to make it a continuous function of $x\in\mathbb{R}$ and we chose the frame of reference by
\begin{equation}
\label{initial}
X^{-1}_0(-1)=1,
\end{equation}
that is to say, the particle labeled $1$ is initially the rightmost in $\mathbb{Z}_{<0}$.

Now, we consider the TASEP that satisfies Assumption \ref{as1}.
Under the above assumption \eqref{Height} and Assumption \ref{imas}, \ref{limcon}, we get the following result for the limiting joint distribution function~\eqref{ldisfunc} by using pointwise convergence of the kernel (see Proposition \ref{scaling}).
\begin{theorem}(One-sided fixed point formula).
\label{special}
We set $\hat{h}_0\in{\rm UC}$ with $\hat{h}_0(\mathbf{x})=-\infty$ for $\mathbf{x}>0$.
Then, for $\mathbf{x}_1<\mathbf{x}_2< \dots <\mathbf{x}_m \in\mathbb{R}$ and $\mathbf{a}_1,\dots, \mathbf{a}_m \in\mathbb{R}$, we obtain
\begin{equation}
\displaystyle\lim_{\varepsilon\xrightarrow{}0}\mathbb{P}_{\hat{h}^{\varepsilon}_0}(\hat{h}^{\varepsilon}(\mathbf{t}, \mathbf{x}_1)\leq\mathbf{a}_1,\dots, \hat{h}^{\varepsilon}(\mathbf{t},\mathbf{x}_m)\leq\mathbf{a}_m)
=
\det\left(\mathbf{I}-\chi_{\mathbf{a}}\mathbf{K}^{{\rm hypo}(\hat{h}_0)}_{\mathbf{t}, {\rm ext}}\chi_{\mathbf{a}}\right)_{L^2(\{\mathbf{x}_1,\dots, \mathbf{x}_m\}\times\mathbb{R})},
\end{equation}
where RHS is equivalent to \eqref{kpzfpdet}.
\end{theorem}
This proof is given in Section \ref{sec5}.
\begin{remark}
In previous studies, the one-sided fixed point formula was obtained only for each model. However, the above results show that the one-sided fixed point formula was obtained for all TASEPs that satisfy Assumption \ref{as1}, \ref{imas}, \ref{limcon}.
\end{remark}

\begin{remark}
By using the similar argument in Theorem 3.8. in~\cite{Quastel}, we can remove the assumption $\hat{h}_0({\bf x})=-\infty$ for ${\bf x}>0$ in Theorem \ref{special} (See subsection 3.4 of \cite{Quastel} for more details.).
\end{remark}
To prove Theorem~\ref{special}, we use the following relationship between the particle positions $X_t(j)$
and the height function $h_t(z)$~\eqref{height}.
We put $s_1,\dots, s_k, m_1, \dots, m_k\in\mathbb{R}$ and $z_1,\dots, z_k, n_1,\dots, n_k\in\mathbb{Z}$.
Then, by Definition \ref{relation1}, we get
\begin{align}
\label{Xhrelation}
\P(h_t(z_1)\le s_1,\dots, h_t(z_k)\le s_k)=\P(X_t(n_1)\ge m_1,\dots, X_t(n_k)\ge m_k).
\end{align}
By the above relation, we see
\begin{multline}
\label{hXrelation}
\lim_{\varepsilon\xrightarrow{}0}\mathbb{P}_{\hat{h}^{\varepsilon}_0}(\hat{h}^{\varepsilon}(\mathbf{t}, \mathbf{x}_1)\leq\mathbf{a}_1,\dots, \hat{h}^{\varepsilon}(\mathbf{t}, \mathbf{x}_m)\leq\mathbf{a}_m)
=
\lim_{\varepsilon\rightarrow0}\displaystyle\mathbb{P}_{X^{\varepsilon}_0}
\left(X^{\varepsilon}_{t}(n_1)>a_1,\dots, X^{\varepsilon}_{t}(n_m)>a_m\right),
\end{multline}
where $a_1,\dots, a_m\in\mathbb{R}$ and $t,~n_i, a_i$ are scaled as
\begin{align}
t=D\varepsilon^{-\frac{3}{2}}\mathbf{t},
~
n_i=
E\varepsilon^{-\frac{3}{2}}\mathbf{t}-\varepsilon^{-1}\mathbf{x}_i
-\frac{1}{2}\varepsilon^{-\frac{1}{2}}\mathbf{a}_i+1,
~
a_i=2\varepsilon^{-1}\mathbf{x}_i-2,
\label{tnascaling}
\end{align}
where $D$ and $E$ are introduced in \eqref{D} and \eqref{E}, respectively.

Therefore our goal \eqref{hXrelation} can be gotten by taking the $\varepsilon\rightarrow 0$
limit of the expression~\eqref{pro} in Theorem~\ref{Main} under the scaling~\eqref{tnascaling}.
The major important step of this problem is the following proposition about pointwise convergence.
\begin{proposition}(Pointwise convergence).
\label{scaling}
We consider the TASEP that satisfies Assumption \ref{as1}.
Suppose that Assumption \ref{imas}, \ref{limcon} and \eqref{varx} hold.
Under the scaling (\ref{tnascaling}),(dropping the $i$ subscripts), if we set $z=G\varepsilon^{-\frac{3}{2}}\mathbf{t}+2\varepsilon^{-1}\mathbf{x}+\varepsilon^{-\frac{1}{2}}(u+\mathbf{a})-2$ and $y'=\varepsilon^{-\frac{1}{2}}v$, then we have for $\mathbf{t}>0$ as $\varepsilon\xrightarrow{} 0$,
\begin{align}
\label{A1}
&\mathbf{S}^{\varepsilon}_{-t,x}(v,u):=\varepsilon^{-\frac{1}{2}}\mathfrak{S}_{-t,-n}(y',z)\xrightarrow{}\mathbf{S}_{-\mathbf{t,x}}(v,u) 
\\
\label{A2}
&\bar{\mathbf{S}}^{\varepsilon}_{-t,-x}(v,u):=\varepsilon^{-\frac{1}{2}}\bar{\mathfrak{S}}_{-t,n}(y',z)\xrightarrow{}\mathbf{S}_{-\mathbf{t},-\mathbf{x}}(v,u)
\\
\label{A3}
&\bar{\mathbf{S}}^{\varepsilon, {\rm epi}(-h^{\varepsilon,-}_0)}_{-t,-x}(v,u):=\varepsilon^{-\frac{1}{2}}\bar{\mathfrak{S}}^{{\rm epi}(X_0)}_{-t,n}(y',z)\xrightarrow{}\mathbf{S}^{{\rm epi}(-\hat{h}^{-}_0)}_{-\mathbf{t},-\mathbf{x}}(v,u) 
\end{align}
pointwise, where
\begin{equation*}
G=\frac{2[\{\gamma^{'}(0)\}^2-\gamma^{(2)}(0)]}{\gamma^{(3)}(0)-3\gamma^{(2)}(0)\gamma^{'}(0)+2\{\gamma^{'}(0)\}^3-2\gamma^{'}(0)},
\end{equation*}
$\hat{h}^{-}_0(x)=\hat{h}_0(-x)$ for $x\geq0$,
$\mathbf{S}_{\mathbf{t}, \mathbf{x}}(v, u)$ is given by~\eqref{function} and for $\hat{g}\in {\rm LC}$,
$$\mathbf{S}^{{\rm epi}(\hat{g})}_{\mathbf{t},\mathbf{x}}(v,u)=\mathbb{E}_{B(0)=v}[\mathbf{S}_{\mathbf{t}, \mathbf{x}-\bm{\tau}'}(B(\bm{\tau}'),u)\1_{\bm{\tau^{'}}<\infty}]$$
and $\mathfrak{S}_{-t,-n}(z_1,z_2)$, $\bar{\mathfrak{S}}_{-t,n}(z_1,z_2)$ and $\bar{\mathfrak{S}}^{{\rm epi}(X_0)}_{-t,n}(z_1, z_2)$ are defined in \eqref{stn}, \eqref{bstn} and \eqref{sepi}, respectively.
\end{proposition}
Proof is given in Section \ref{sec5}.
\section{Distribution representation for the TASEP }
\label{sec3}
\subsection{The representation of the transition probability for TASEP: Proof of Proposition \ref{m13} and Theorem \ref{s88}}
\label{s31}
In this subsection, we show that the transition probability can be expressed by a probability generating function of the distribution followed when the rightmost particle jumps.
First, we see that the property of the transition probability when the number of particles is one.
\begin{lemma}
\label{m4}
Suppose that Assumption \ref{as1} holds.
When $N=1$, for $x, y\in\mathbb{Z}$ such that $x\geq y$, 
\begin{equation}
\label{s23}
\mathbb{P}(X_t(1)=x| X_0(1)=y)=F_0(x-y, t)
\end{equation}
\end{lemma}
\begin{proof}
It is easy to see that \eqref{s23} holds if $N=1$, $x_1=x$ and $y_1=y$ are substituted in \eqref{s22}.
\end{proof}
Next, we prove the space-homogeneity when $N=1$.
\begin{lemma}[space-homogeneity]
\label{s24}
Assume that Assumption \ref{as1} holds.
When $N=1$, for $x, y\in\mathbb{Z}$ such that $x\geq y$, 
\begin{equation*}
\mathbb{P}(X_t(1)=x| X_0(1)=y)=\mathbb{P}(X_t(1)=x-y| X_0(1)=0).
\end{equation*}
\end{lemma}
\begin{proof}
By Lemma \ref{m4}, we get 
\begin{equation*}
\begin{split}
\mathbb{P}(X_t(1)=x| X_0(1)=y)&=F_0(x-y, t)\\
&=\mathbb{P}(X_t(1)=x-y| X_0(1)=0).
\end{split}
\end{equation*}
\end{proof}
Note that when $N=1$, the exclusion rule doesn't work, so the TASEP is just a one-sided jump random walk or a compound Poisson process.
For convenience, we set the following:
\begin{equation}
\label{benri}
\mathbb{P}(X_t(1)=x):=\mathbb{P}(X_t(1)=x| X_0(1)=0).
\end{equation}
Now, we show that the following holds.
\begin{lemma}
\label{mm8}
We consider the TASEP that satisfies Assumption \ref{as1}. Then the following two are equivalent:
\begin{enumerate}
\item $\displaystyle\mathcal{M}(t, w)=\sum_{x=0}^{\infty}w^x\mathbb{P}(X_t(1)=x)$ where radius of convergence is $R\geq 1$.
\item $F_0(x, t)=\mathbb{P}(X_t(1)=x).$
\end{enumerate}
\end{lemma}
\begin{proof}
First we prove $(\hspace{.01em}i\hspace{.01em})\Rightarrow(\hspace{.01em}ii\hspace{.01em})$.
By $\eqref{s21}$, 
\begin{equation}
\label{m7}
F_0(x,t)=\frac{1}{2\pi i}\oint_{\Gamma_{0}}dw\frac{1}{w^{x+1}}\mathcal{M}(t, w)
\end{equation}
where $\Gamma_{0}$ is any simple loop oriented anticlockwise and includes $w=0$.
By changing variables as $w=re^{iv}$ when we define $r\in\mathbb{R}$ such that $0<r<R$, 
\begin{equation}
\begin{split}
\label{m8}
\eqref{m7}&=\frac{1}{2\pi i}\int_{0}^{2\pi}dv \frac{1}{(re^{iv})^{x+1}}ire^{iv}\mathcal{M}(t, re^{iv})\\
&=\frac{1}{2\pi}\int_{0}^{2\pi}dv r^{-x}e^{-ixv}\mathcal{M}(t, re^{iv}).
\end{split}
\end{equation}
Since
$$\mathcal{M}(t, re^{iv})=\sum_{k=0}^{\infty}r^ke^{ikv}\mathbb{P}(X_t(1)=k)$$
is absolutely convergent,
\begin{equation}
\label{m9}
\begin{split}
\eqref{m8}&=\frac{1}{2\pi}\int_{0}^{2\pi}dv r^{-x}e^{-ixv}\sum_{k=0}^{\infty}r^ke^{ikv}\mathbb{P}(X_t(1)=k)\\
&=\frac{1}{2\pi}\sum_{k=0}^{\infty}r^{k-x}\mathbb{P}(X_t(1)=k)\int_{0}^{2\pi}e^{i(k-x)v}dv\\
&=\frac{1}{2\pi}\sum_{k=0}^{\infty}r^{k-x}\mathbb{P}(X_t(1)=k)2\pi\delta_{k,x}\\
&=\mathbb{P}(X_t(1)=x)
\end{split}
\end{equation}
where $\delta_{k, x}$ is the Kronecker delta, that is,
\begin{equation}
\label{benri2}
\delta_{k,x}:=
\begin{cases}
1 & \text{for $k=x$,}\\
0 & \text{otherwise.}
\end{cases}
\end{equation}
Next we show $(\hspace{.01em}ii\hspace{.01em})\Rightarrow(\hspace{.01em}i\hspace{.01em})$.
For $s\in\mathbb{C}$, we put 
\begin{equation*}
\phi(s)=\sum_{x=0}^{\infty}s^x \mathbb{P}(X_t(1)=x)
\end{equation*}
where radius of convergence is $R\geq 1$.
From the condition $(\hspace{.01em}ii\hspace{.01em})$,
\begin{equation}
\label{m10}
\phi(s)=\sum_{x=0}^{\infty}F_0(x, t) s^x
\end{equation}
holds. 
By substituting \eqref{m7} for \eqref{m10}, we obtain 
\begin{equation}
\begin{split}
\phi(s)&=\sum_{x=0}^{\infty}\left\{\frac{1}{2\pi i}\oint_{\Gamma_{0}}dw\frac{1}{w^{x+1}}\mathcal{M}(t, w)\right\}s^x\\
&=\frac{1}{2\pi i}\oint_{\Gamma_{0}}dw\frac{\mathcal{M}(t, w)}{w}\sum_{x=0}^{\infty}\left(\frac{s}{w}\right)^x\\
&=\frac{1}{2\pi i}\oint_{\Gamma_{0}}dw\frac{\mathcal{M}(t, w)}{w}\frac{1}{1-\frac{s}{w}}\\
&=\frac{1}{2\pi i}\oint_{\Gamma_{0}}dw\frac{\mathcal{M}(t, w)}{w-s}\\
&=\mathcal{M}(t, s)
\end{split}
\end{equation}
provided that the integration domain satisfy $|s|<|w|$.
Therefore, we get
\begin{equation*}
\mathcal{M}(t, w)=\sum_{x=0}^{\infty}w^x\mathbb{P}(X_t(1)=x)
\end{equation*}
where radius of convergence is $R\geq 1$.
\end{proof}
\begin{remark}
By \eqref{s22}, Lemma \ref{m4} and Lemma \ref{s24}, it can be seen that $(\mathrm{ii})~F_0(x, t)=\mathbb{P}(X_t(1)=x)$ holds.
Thus, Lemma \ref{mm8} implies that  the function $\mathcal{M}(t, w)$ which constitutes the transition probability of TASEP is given as the probability generating function.
\end{remark}
Now, we prove Proposition \ref{m13}.
\begin{proof}[Proof of Proposition \ref{m13}]
By Lemma \ref{m4}, Lemma \ref{s24} and Lemma \ref{mm8},
\begin{equation*}
\mathcal{M}(t, w)=\sum_{x=0}^{\infty}w^x\mathbb{P}(X_t(1)=x).
\end{equation*}
Now, we first show the case of $t\in\mathbb{Z}_{\geq 0}$.
By \eqref{m11}, 
\begin{equation}
\label{m14}
\begin{split}
\mathcal{M}(t, w)&=\mathbb{E}\left[w^{X_t(1)}\right]\\
&=\mathbb{E}\left[w^{Y_1+Y_2+\dots+Y_t}\right]\\
&=\left(\mathbb{E}\left[w^{Y_1}\right]\right)^t\\
&=\mathcal{M}(w)^t.
\end{split}
\end{equation}
Next, we prove the case of $t\in\mathbb{R}_{\geq 0}$.
By \eqref{m12} and \eqref{jyuuyou5},
\begin{equation*}
\begin{split}
\mathcal{M}(t, w)&=\mathbb{E}[w^{S_{N_t}}]\\
&=\sum_{n=0}^{\infty}\mathbb{E}[w^{S_{N_t}}| N_t=n]\mathbb{P}(N_t=n)\\
&=\sum_{n=0}^{\infty}\mathbb{E}[w^{S_{n}}]\mathbb{P}(N_t=n)\\
&=\sum_{n=0}^{\infty}\mathbb{E}[w^{Z_1+Z_2+\dots+Z_{n}}]\mathbb{P}(N_t=n)\\
&=\sum_{n=0}^{\infty}(\mathbb{E}[w^{Z_1}])^n \mathbb{P}(N_t=n)\\
&=\sum_{n=0}^{\infty}\{G_{Z_1}(w)\}^ne^{-\lambda t}\frac{(\lambda t)^n}{n!}\\
&=e^{\lambda t\{G_{Z_1}(w)-1\}}.
\end{split}
\end{equation*}
On the other hand, by \eqref{pgf},
\begin{equation}
\label{mtw2}
\begin{split}
\mathcal{M}(w)&=G_{X}(G_{Z_1}(w))\\
&=\sum_{n=0}^{\infty}\{G_{Z_1}(w)\}^ne^{-\lambda}\frac{\lambda^n}{n!}\\
&=e^{\lambda \{G_{Z_1}(w)-1\}}.
\end{split}
\end{equation}
Therefore we get 
$$\mathcal{M}(t, w)=\mathcal{M}(w)^t.$$
\end{proof}
Next, we show Theorem \ref{s88}.
\begin{proof}[Proof of Theorem \ref{s88}]
By \eqref{s21} and Proposition \ref{m13}, 
\begin{equation*}
\begin{split}
F_n(x,t)&=\frac{(-1)^n}{2\pi i}\oint_{\Gamma_{0,1}}dw\frac{(1-w)^{-n}}{w^{x-n+1}}\mathcal{M}(t, w)\\
&=\frac{(-1)^n}{2\pi i}\oint_{\Gamma_{0,1}}dw\frac{(1-w)^{-n}}{w^{x-n+1}}\mathcal{M}(w)^t\\
&=\overline{F}_n(x,t).
\end{split}
\end{equation*}
From the above, the transition probability of TASEP is given by
\begin{equation*}
\mathbb{P}(X_t=\vec{x}|X_0=\vec{y})=\det[\overline{F}_{i-j}(x_{N+1-i}-y_{N+1-j}, t)]_{1\leq i,j\leq N}.
\end{equation*}
\end{proof}
\subsection{Proof of Theorem \ref{Main}}
\label{s32}
In this subsection, we prove Theorem \ref{Main}.

It is easy to check that $\mathcal{M}(w)$ satisfies Assumption 1.1 in \cite{Quastel}.
Therefore, by Theorem 1.2 in \cite{Quastel} and Theorem \ref{s88}, we have the distribution of particle positions:
\begin{equation*}
\mathbb{P}(X_t(n_j)>a_j, j=1,\dots,M)=\det(I-\bar{\chi}_{a}K_t\bar{\chi}_{a})_{\ell^2(\{n_1,\dots,n_M\}\times\mathbb{Z})}
\end{equation*}
where
\begin{equation}
K_t(n_i, x  ; n_j, y)=-Q^{n_j-n_i}(x,y)\1_{n_i<n_j}+(\mathscr{S}_{-t, -n_i})^{*}\overline{\mathscr{S}}^{{\rm epi} (X_0)}_{-t,n_j}(x,y)
\end{equation}
and
\begin{equation*}
\mathscr{S}_{-t,-n}(z_1,z_2)=\frac{1}{2\pi i}\oint_{\Gamma_0}dw\frac{(1-w)^n}{2^{z_2-z_1}w^{n+1+z_2-z_1}}\mathcal{M}(w)^{t},
\end{equation*}
\begin{equation*}
\overline{\mathscr{S}}_{-t,n}(z_1,z_2)=\frac{1}{2\pi i}\oint_{\Gamma_0}dw\frac{(1-w)^{z_2-z_1+n-1}}{2^{z_1-z_2}w^{n}}\mathcal{M}(1-w)^{-t},
\end{equation*}
\begin{equation*}
\overline{\mathscr{S}}^{\rm epi(X_0)}_{-t,n}(z_1, z_2)=\mathbb{E}_{RW_0=z_1}\left[\,\overline{\mathscr{S}}_{-t, n-\tau}(RW_{\tau}, z_2)\1_{\tau<n}\right].
\end{equation*}
From \eqref{pro} and \eqref{Kt}, it is enough to check
\begin{equation*}
(\mathfrak{S}_{-t, -n_i})^{*}\bar{\mathfrak{S}}^{{\rm epi} (X_0)}_{-t,n_j}=(\mathscr{S}_{-t, -n_i})^{*}\overline{\mathscr{S}}^{{\rm epi} (X_0)}_{-t,n_j}
\end{equation*}
to prove Theorem \ref{Main}.
Because
\begin{equation*}
\mathfrak{S}_{-t,-n}(z_1,z_2)=\mathcal{M}\left(\frac{1}{2}\right)^{-t}\mathscr{S}_{-t,-n}(z_1,z_2)
\end{equation*}
and
\begin{equation*}
\overline{\mathfrak{S}}_{-t,n}(z_1,z_2)=\mathcal{M}\left(\frac{1}{2}\right)^{t}\overline{\mathscr{S}}_{-t,n}(z_1,z_2),
\end{equation*}
we get 
\begin{equation*}
(\mathfrak{S}_{-t, -n_i})^{*}\bar{\mathfrak{S}}^{{\rm epi} (X_0)}_{-t,n_j}=(\mathscr{S}_{-t, -n_i})^{*}\overline{\mathscr{S}}^{{\rm epi} (X_0)}_{-t,n_j}.
\end{equation*}
This completes the proof.
\section{The property of the coefficient of KPZ scaling : Proof of Theorem \ref{main2}}
\label{sec4}
In this section we prove Theorem \ref{main2}.
First, we show the case of $t\in\mathbb{Z}_{\geq 0}$.
Note that $\mathcal{M}(w)$ and $\mathcal{M}^{(n)}(w)$ have the same radius of convergence by the Cauchy-Hadamard theorem where $\mathcal{M}^{(n)}(w)$ is the n-th derivative of $\mathcal{M}(w)$. 
Because $\mathcal{M}(w)$ is absolutely convergent for $|w|\leq 1$, $\mathcal{M}^{(n)}(w)$ is absolutely convergent for $|w|\leq 1$.
Thus we get
\begin{equation*}
|\gamma^{(2)}(0)-\{\gamma^{'}(0)\}^2-\gamma^{'}(0)|<\infty
\end{equation*}
and
\begin{equation*}
|\gamma^{(3)}(0)-3\gamma^{(2)}(0)\gamma^{'}(0)+2\{\gamma^{'}(0)\}^3-2\gamma^{'}(0)|<\infty.
\end{equation*}

Now, we remark that
\begin{equation*}
\gamma^{(n)}(0)=(-1)^n\frac{\mathbb{E}[Y_1(Y_1-1)\dots(Y_1-n+1)(\frac{1}{2})^{Y_1}]}{\mathcal{M}(\frac{1}{2})}.
\end{equation*}
Using Cauchy-Schwarz inequality, we have
\begin{equation*}
\begin{split}
\gamma^{(2)}(0)-\{\gamma^{'}(0)\}^2-\gamma^{'}(0)&=\frac{\mathbb{E}[Y_1(Y_1-1)(\frac{1}{2})^{Y_1}]}{\mathcal{M}(\frac{1}{2})}-\left\{-\frac{\mathbb{E}[Y_1(\frac{1}{2})^{Y_1}]}{\mathcal{M}(\frac{1}{2})}\right\}^2+\frac{\mathbb{E}[Y_1(\frac{1}{2})^{Y_1}]}{\mathcal{M}(\frac{1}{2})}\\
&=\frac{\mathbb{E}[Y^2_1(\frac{1}{2})^{Y_1}]\mathbb{E}[(\frac{1}{2})^{Y_1}]-\{\mathbb{E}[Y_1(\frac{1}{2})^{Y_1}]\}^2}{\{\mathcal{M}(\frac{1}{2})\}^2}\geq 0.
\end{split}
\end{equation*}
Next, we prove the case of $t\in\mathbb{R}_{\geq 0}$.
Recall 
\begin{equation*}
\mathcal{M}(w)=e^{\lambda \{G_{Z_1}(w)-1\}}
\end{equation*}
from \eqref{mtw2}.
Note that $G_{Z_1}(w)$ and $G^{(n)}_{Z_1}(w)$ have the same radius of convergence by the Cauchy-Hadamard theorem where $G^{(n)}_{Z_1}(w)$ is the n-th derivative of $G_{Z_1}(w)$.
Since $G_{Z_1}(w)$ is absolutely convergent for $|w|\leq 1$, $G^{(n)}_{Z_1}(w)$ is absolutely convergent for $|w|\leq 1$.
Therefore, for $|1-w|\leq 2$, the derivatives of $\mathcal{M}(\frac{1}{2}(1-w))$ up to the third order are
\begin{equation}
\label{vbn1}
\frac{d}{dw}\left[\mathcal{M}\left(\frac{1}{2}(1-w)\right)\right]=\lambda\left\{\frac{d}{dw}\left[G_{Z_1}\left(\frac{1}{2}(1-w)\right)\right]\right\}e^{\lambda\{G_{Z_1}(\frac{1}{2}(1-w))-1\}},
\end{equation}
\begin{equation}
\label{vbn2}
\begin{split}
\frac{d^2}{dw^2}\left[\mathcal{M}\left(\frac{1}{2}(1-w)\right)\right]&=\lambda\left\{\frac{d^2}{dw^2}\left[G_{Z_1}\left(\frac{1}{2}(1-w)\right)\right]\right\}e^{\lambda\{G_{Z_1}(\frac{1}{2}(1-w))-1\}}\\
&+\lambda^2\left\{\frac{d}{dw}\left[G_{Z_1}\left(\frac{1}{2}(1-w)\right)\right]\right\}^2e^{\lambda\{G_{Z_1}(\frac{1}{2}(1-w))-1\}},
\end{split}
\end{equation}
\begin{equation}
\label{vbn3}
\begin{split}
\frac{d^3}{dw^3}\left[\mathcal{M}\left(\frac{1}{2}(1-w)\right)\right]&=\lambda\left\{\frac{d^3}{dw^3}\left[G_{Z_1}\left(\frac{1}{2}(1-w)\right)\right]\right\}e^{\lambda\{G_{Z_1}(\frac{1}{2}(1-w))-1\}}\\
&+3\lambda^2\left\{\frac{d^2}{dw^2}\left[G_{Z_1}\left(\frac{1}{2}(1-w)\right)\right]\right\}\left\{\frac{d}{dw}\left[G_{Z_1}\left(\frac{1}{2}(1-w)\right)\right]\right\}e^{\lambda\{G_{Z_1}(\frac{1}{2}(1-w))-1\}}\\
&+\lambda^3\left\{\frac{d}{dw}\left[G_{Z_1}\left(\frac{1}{2}(1-w)\right)\right]\right\}^3e^{\lambda\{G_{Z_1}(\frac{1}{2}(1-w))-1\}}.
\end{split}
\end{equation}
By \eqref{gamgam}, \eqref{vbn1}, \eqref{vbn2} and \eqref{vbn3}, we get
\begin{equation}
\label{wbn1}
\gamma^{'}(0)=\lambda\left\{\left.\frac{d}{dw}\left[G_{Z_1}\left(\frac{1}{2}(1-w)\right)\right]\right|_{w=0}\right\},
\end{equation}
\begin{equation}
\label{wbn2}
\gamma^{(2)}(0)=\lambda\left\{\left.\frac{d^2}{dw^2}\left[G_{Z_1}\left(\frac{1}{2}(1-w)\right)\right]\right|_{w=0}\right\}+\lambda^2\left\{\left.\frac{d}{dw}\left[G_{Z_1}\left(\frac{1}{2}(1-w)\right)\right]\right|_{w=0}\right\}^2
\end{equation}
and
\begin{equation*}
\begin{split}
\gamma^{(3)}(0)&=\lambda\left\{\left.\frac{d^3}{dw^3}\left[G_{Z_1}\left(\frac{1}{2}(1-w)\right)\right]\right|_{w=0}\right\}\\
&+3\lambda^2\left\{\left.\frac{d^2}{dw^2}\left[G_{Z_1}\left(\frac{1}{2}(1-w)\right)\right]\right|_{w=0}\right\}\left\{\left.\frac{d}{dw}\left[G_{Z_1}\left(\frac{1}{2}(1-w)\right)\right]\right|_{w=0}\right\}\\
&+\lambda^3\left\{\left.\frac{d}{dw}\left[G_{Z_1}\left(\frac{1}{2}(1-w)\right)\right]\right|_{w=0}\right\}^3.
\end{split}
\end{equation*}
Because $\left.\frac{d^n}{dw^n}\left[G_{Z_1}\left(\frac{1}{2}(1-w)\right)\right]\right|_{w=0}$ is absolutely convergent, we have
\begin{equation*}
|\gamma^{(2)}(0)-\{\gamma^{'}(0)\}^2-\gamma^{'}(0)|<\infty
\end{equation*}
and
\begin{equation*}
|\gamma^{(3)}(0)-3\gamma^{(2)}(0)\gamma^{'}(0)+2\{\gamma^{'}(0)\}^3-2\gamma^{'}(0)|<\infty.
\end{equation*}
Thus, by \eqref{wbn1} and \eqref{wbn2}, we get
\begin{equation*}
\begin{split}
\gamma^{(2)}(0)-\{\gamma^{'}(0)\}^2-\gamma^{'}(0)&=\lambda\left\{\left.\frac{d^2}{dw^2}\left[G_{Z_1}\left(\frac{1}{2}(1-w)\right)\right]\right|_{w=0}-\left.\frac{d}{dw}\left[G_{Z_1}\left(\frac{1}{2}(1-w)\right)\right]\right|_{w=0}\right\}\\
&=\lambda\left\{\mathbb{E}\left[Z_1(Z_1-1)\left(\frac{1}{2}\right)^{Z_1}\right]+\mathbb{E}\left[Z_1\left(\frac{1}{2}\right)^{Z_1}\right]\right\}\\
&=\lambda~\mathbb{E}\left[Z^2_1\left(\frac{1}{2}\right)^{Z_1}\right]\geq 0.
\end{split}
\end{equation*}
This completes the proof.
\section{Asymptotics}
\label{sec5}
In this section, given Assumptions \ref{imas} and Assumption \ref{limcon}, we take the KPZ scaling limit for the TASEP which satisfies Assumption \ref{as1} and prove Proposition \ref{scaling}.
After that, we prove Theorem \ref{special} by using Proposition \ref{scaling}.
\subsection{Proof of Proposition \ref{scaling}}
First, we show \eqref{A1}.
By changing variables $\displaystyle w=\frac{1}{2}(1-\varepsilon^{\frac{1}{2}}y)$, we get
\begin{equation}
\begin{split}
\label{e}
\mathbf{S}^{\varepsilon}_{-t,x}(v,u)&=\frac{1}{2\pi i}\oint_{C_{\varepsilon}}dy\frac{\{\frac{1}{2}(1+\varepsilon^{\frac{1}{2}}y)\}^n}{2^{z-y^{'}+1}\{\frac{1}{2}(1-\varepsilon^{\frac{1}{2}}y)\}^{n+1+z-y'}}\left\{\frac{\mathcal{M}(\frac{1}{2}(1-\varepsilon^{\frac{1}{2}}y))}{\mathcal{M}\left(\frac{1}{2}\right)}\right\}^{t}\\
&=\frac{1}{2\pi i}\oint_{C_{\varepsilon}}dy\frac{(1+\varepsilon^{\frac{1}{2}}y)^n}{(1-\varepsilon^{\frac{1}{2}}y)^{n+1+z-y'}}\gamma\left(\varepsilon^{\frac{1}{2}}y\right)^t
\end{split}
\end{equation}
where $C_{\varepsilon}$ is a circle of radius $\varepsilon^{-\frac{1}{2}}$ centred at $\varepsilon^{-\frac{1}{2}}$.
So as to use the saddle point method, we rewrite~\eqref{e} as
\begin{equation}
\label{fF3}
\frac{1}{2\pi i}\oint_{C_{\varepsilon}}e^{f(\varepsilon^{\frac{1}{2}}y)+\varepsilon^{-1}F_2(\varepsilon^{\frac{1}{2}}y)+\varepsilon^{-\frac{1}{2}}F_1(\varepsilon^{\frac{1}{2}}y)+F_0(\varepsilon^{\frac{1}{2}}y)}dy,
\end{equation}
where the functions $f(x)$ and $F_i(x),~i=0,1,2$ are defined by
\begin{align}
&f(x)=E\hat{t}\log(1+x)-F
\hat{t}\log(1-x)+D\hat{t}\log\gamma(x),
\\
&F_2(x)=-\mathbf{x}\log(1-x^2),~F_1(x)=(v-u-\frac{1}{2}\mathbf{a})\log(1-x)-\frac{1}{2}\mathbf{a}\log(1+x),~ F_0(x):=\log(1+x)
\end{align}
where $\hat{t}:=\varepsilon^{-\frac{3}{2}}\mathbf{t}$ and $D$, $E$ and $F$ are introduced in \eqref{D}, \eqref{E} and \eqref{F}, respectively.
Calculating the derivatives of $f(x)$ up to the third order, we obtain
\begin{equation*}
f^{'}(x)=E\hat{t}\frac{1}{1+x}+F\hat{t}\frac{1}{1-x}+D\hat{t}\frac{\gamma^{'}(x)}{\gamma(x)},
\end{equation*}
\begin{equation*}
f^{(2)}(x)=-E\hat{t}\frac{1}{(1+x)^2}+F\hat{t}\frac{1}{(1-x)^2}+D\hat{t}\frac{\gamma^{(2)}(x)\gamma(x)-\{\gamma^{'}(x)\}^2}{\{\gamma(x)\}^2},
\end{equation*}
\begin{equation*}
f^{(3)}(x)=2E\hat{t}\frac{1}{(1+x)^3}+2F\hat{t}\frac{1}{(1-x)^3}+D\hat{t}\frac{\gamma^{(3)}(x)\{\gamma(x)\}^2-3\gamma^{(2)}(x)\gamma^{'}(x)\gamma(x)+2\{\gamma^{'}(x)\}^3}{\{\gamma(x)\}^3}.
\end{equation*}
Therefore, we find $f(x)$ has the double saddle point at $x=0$, 
\begin{equation*}
 f(0)=0, \ f'(0)=0, \ f^{(2)}(0)=0 \ {\rm and} \ f^{(3)}(0)=2\hat{t}.
\end{equation*}
Thus, for small $\varepsilon$, we obtain
\begin{equation}
\label{fex}
f(\varepsilon^{\frac{1}{2}}y)\approx\frac{\mathbf{t}}{3}y^3.
\end{equation}
For $F_i(x),~i=0,1,2$, we easily see
\begin{equation}
\label{Fasy}
\varepsilon^{-1}F_2(\varepsilon^{\frac{1}{2}}y)\approx\mathbf{x}y^2, \ \varepsilon^{-\frac{1}{2}}F_1(\varepsilon^{\frac{1}{2}}y)\approx(u-v)y, \ F_0(\varepsilon^{\frac{1}{2}}y)\approx0.
\end{equation}
Now, we check the convergence of the integration path.
We divide $C_{\varepsilon}$ into two parts $\langle_{\varepsilon} \  \cup \  C^{\frac{\pi}{3}}_{\varepsilon}$ where $\langle_{\varepsilon}$ is the part of Airy contour $\langle$ within the ball of radius $\varepsilon^{-\frac{1}{2}}$ centred at $\varepsilon^{-\frac{1}{2}}$, and $C^{\frac{\pi}{3}}_{\varepsilon}$ is the part of $C_{\varepsilon}$ to the right of $\langle$. 
Note that $\mathbf{S}_{\mathbf{t}, \mathbf{x}}(u, v)=\mathbf{S}_{\mathbf{-t}, \mathbf{x}}(v, u)$  where $\mathbf{S}_{\mathbf{t}, \mathbf{x}}(u, v)$ is introduced by~\eqref{function}.
By~\eqref{fex}, and~\eqref{Fasy}, we get
\begin{equation*}
\lim_{\e\rightarrow 0}\frac{1}{2\pi i}\int_{\langle_\varepsilon} e^{f(\varepsilon^{\frac{1}{2}}y)+\varepsilon^{-1}F_2(\varepsilon^{\frac{1}{2}}y)+\varepsilon^{-\frac{1}{2}}F_1(\varepsilon^{\frac{1}{2}}y)+F_0(\varepsilon^{\frac{1}{2}}y)}dy=\mathbf{S}_{\mathbf{-t}, \mathbf{x}}(v, u)
\end{equation*}

Next, we prove that the part coming from $C^{\frac{\pi}{3}}_{\varepsilon}$ vanishes as $\varepsilon\rightarrow0$
To see this note that the real part of the exponent of the integral over $C_{\varepsilon}$ in (\ref{e}), parametrized as $y=\varepsilon^{-\frac{1}{2}}(1-e^{i\theta})$, is given by 
\begin{equation*}
\frac{1}{2}\varepsilon^{-\frac{3}{2}}\mathbf{t}\left[\left(E+\mathcal{O}(\varepsilon^{\frac{1}{2}})\right)\log|2-e^{i\theta}|+D\log|\gamma(1-e^{i\theta})|\right].
\end{equation*}
Note that the $y\in C^{\frac{\pi}{3}}_{\varepsilon}$ correspond to $\frac{\pi}{3}<|\theta|\leq\pi$.
From \eqref{condi10} of Assumption \ref{limcon},
\begin{equation*}
\frac{1}{2}\varepsilon^{-\frac{3}{2}}\mathbf{t}\left[E\log|2-e^{i\theta}|+D\log|\gamma(1-e^{i\theta})|\right]<0.
\end{equation*}
Thus, for sufficiently small $\varepsilon$, the exponent there is less than $-\varepsilon^{-\frac{3}{2}}\kappa\mathbf{t}$ for some $\kappa>0$.
Hence the part $C_{\e}^{\frac{\pi}{3}}$ of the integral vanishes and this completes the proof of~\eqref{A1}.

Next, we prove \eqref{A2}.
By changing variables $\displaystyle w=\frac{1}{2}(1-\varepsilon^{\frac{1}{2}}y)$, 
\begin{equation*}
\begin{split}
\bar{\mathbf{S}}^{\varepsilon}_{-t,-x}(v,u)&=\frac{1}{2\pi i}\oint_{C_{\varepsilon}}dy\frac{\{\frac{1}{2}(1+\varepsilon^{\frac{1}{2}}y)\}^{z-y'+n-1}}{2^{y'-z+1}\{\frac{1}{2}(1-\varepsilon^{\frac{1}{2}}y)\}^{n}}\left\{\frac{\mathcal{M}(\frac{1}{2}(1+\varepsilon^{\frac{1}{2}}y))}{\mathcal{M}\left(\frac{1}{2}\right)}\right\}^{-t}\\
&=\frac{1}{2\pi i}\oint_{C_{\varepsilon}}dy\frac{(1+\varepsilon^{\frac{1}{2}}y)^{z-y'+n-1}}{(1-\varepsilon^{\frac{1}{2}}y)^{n}}\overline{\gamma}\left(\varepsilon^{\frac{1}{2}}y\right)^{-t}
\end{split}
\end{equation*}
where $C_{\varepsilon}$ is a circle of radius $\varepsilon^{-\frac{1}{2}}$ centred at $\varepsilon^{-\frac{1}{2}}$ and
\begin{equation*}
\overline{\gamma}(w)=\frac{\mathcal{M}(\frac{1}{2}(1+w))}{\mathcal{M}\left(\frac{1}{2}\right)}.
\end{equation*}
We remark that
\begin{equation*}
\overline{\gamma}(0)=\gamma(0),~\overline{\gamma}^{'}(0)=-\gamma^{'}(0),~\overline{\gamma}^{(2)}(0)=\gamma^{(2)}(0),~\overline{\gamma}^{(3)}(0)=-\gamma^{(3)}(0).
\end{equation*}
Using the saddle point method and \eqref{condi11} of Assumption \ref{limcon}, we can also show \eqref{A2} in the similar way to \eqref{A1}.

For \eqref{A3}, we can prove it similarly to (3.17) of Lemma 3.5 in \cite{Quastel}.
Therefore, only a sketch of the proof is given here.
We set the scaled walk $\displaystyle\mathbf{B}^{\varepsilon}(x)=\varepsilon^{\frac{1}{2}}(\text{RW}_{\varepsilon^{-1}x}+2\varepsilon^{-1}x-1)$ for $x\in\varepsilon\mathbb{Z}_{\geq0}$, interpolated linearly in between, and put $\bm\tau^{\varepsilon}$ as the hitting time by $\mathbf{B}^{\varepsilon}$ of ${\rm epi}(-\hat{h}^{\varepsilon}(0,\cdot)^{-})$
where $\hat{h}^{\varepsilon}({\bf t, x})$ is introduced by~\eqref{Hei} and $\hat{h}^{\varepsilon}({\bf t, x})^-=\hat{h}^{\varepsilon}({\bf t, -x})$.
By Donsker's invariance principle \cite{Patrick}, $\displaystyle\mathbf{B}^{\varepsilon}(x)$ converges locally uniformly in distribution to a Brownian motion $\mathbf{B}(x)$ with diffusion coefficient $2$.
Also, combining this with (\ref{varx}) and Proposition 3.2 in \cite{Quastel}, we see that the hitting time $\bm\tau^{\varepsilon}$ converges to $\bm\tau$. (For more detail,
see Lemma 3.5 in \cite{Quastel}).)
This leads to (\ref{A3}).
\subsection{Proof of Theorem \ref{special}}
\label{sspf}
In this subsection, we show Theorem \ref{special} by using Propositions \ref{scaling}.
This proof is almost the same as Proposition 3.6 in \cite{Quastel}.
First, we change variables in the kernel as in Proposition \ref{scaling}.
Then, for $z_i=G\varepsilon^{-\frac{3}{2}}\mathbf{t}+2\varepsilon^{-1}\mathbf{x}_i+\varepsilon^{-\frac{1}{2}}(u_i+\mathbf{a}_i)-2$, we compute the limiting kernel
\begin{equation*}
\mathbf{K}_{\lim}(\mathbf{x}_i, u_i ; \mathbf{x}_j, u_j):=\lim_{\varepsilon\xrightarrow{}0}\varepsilon^{-\frac{1}{2}}(\bar{\chi}_{2\varepsilon^{-1}\mathbf{x}-2}K_t \bar{\chi}_{2\varepsilon^{-1}\mathbf{x}-2})(z_i, z_j)
\end{equation*}
where
\begin{equation*}
G=\frac{2[\{\gamma^{'}(0)\}^2-\gamma^{(2)}(0)]}{\gamma^{(3)}(0)-3\gamma^{(2)}(0)\gamma^{'}(0)+2\{\gamma^{'}(0)\}^3-2\gamma^{'}(0)}
\end{equation*}
and  $\gamma(w)$ is defined in \eqref{gamgam}.
We remark that the change of variables turns $\bar{\chi}_{2\varepsilon^{-1}\mathbf{x}-2}(z)$ into $\bar{\chi}_{-\mathbf{a}}(u)$.
We obtain $n_i<n_j$ for small $\varepsilon$ if and only if $\mathbf{x}_j<\mathbf{x}_i$ and in this case we get
\begin{equation}
\label{kernel22}
\lim_{\varepsilon\xrightarrow{}0}\varepsilon^{-\frac{1}{2}}Q^{n_j-n_i}(z_i,z_j)=e^{(\mathbf{x}_i-\mathbf{x}_j)\partial^2}(u_i,u_j).
\end{equation}
For the second term in (\ref{Kt}), we have
\begin{equation*}
\begin{split}
\varepsilon^{-\frac{1}{2}}(\mathfrak{S}_{-t, -n_i})^{*}\bar{\mathfrak{S}}^{\rm epi(X_0)}_{-t,n_j}(z_i,z_j)&=\varepsilon^{-\frac{1}{2}}\int_{-\infty}^{\infty}d\nu(\mathfrak{S}_{-t, -n_i})^{*}(z_i,\nu)\bar{\mathfrak{S}}^{\rm epi(X_0)}_{-t,n_j}(\nu,z_j)\\
&=\varepsilon^{-1}\int_{-\infty}^{\infty}d\nu(\mathfrak{S}_{-t, -n_i})^{*}(z_i,\varepsilon^{-\frac{1}{2}}\nu)\bar{\mathfrak{S}}^{\rm epi(X_0)}_{-t,n_j}(\varepsilon^{-\frac{1}{2}}\nu,z_j)\\
&=\int_{-\infty}^{\infty}d\nu (\mathbf{S}^{\varepsilon}_{-t,x_i})^{*}(u_i,\nu)\bar{\mathbf{S}}^{\varepsilon, \rm{epi}(-h^{\varepsilon,-}_0)}_{-t,-x_j}(\nu,u_j)\\
&=(\mathbf{S}^{\varepsilon}_{-t,x_i})^{*}\bar{\mathbf{S}}^{\varepsilon, \rm{epi}(-h^{\varepsilon,-}_0)}_{-t,-x_j}(u_i,u_j).
\end{split}
\end{equation*}
By Proposition \ref{scaling}, we obtain
\begin{equation}
\label{kernel23}
\lim_{\varepsilon\xrightarrow{}0}\varepsilon^{-\frac{1}{2}}(\mathfrak{S}_{-t, -n_i})^{*}\bar{\mathfrak{S}}^{\rm epi(X_0)}_{-t,n_j}(z_i,z_j)=(\mathbf{S}_{-\mathbf{t}, \mathbf{x}_i})^{*}\mathbf{S}^{\rm{epi}(-\hat{h}^{-}_0)}_{-\mathbf{t},-\mathbf{x}_j}(u_i,u_j).
\end{equation}
By \eqref{kernel22} and \eqref{kernel23}, we get
\begin{equation*}
\mathbf{K}_{\lim}(\mathbf{x}_i, u_i ; \mathbf{x}_j, u_j)=-e^{(\mathbf{x}_i-\mathbf{x}_j)\partial^2}(u_i,u_j)\1_{\mathbf{x}_i>\mathbf{x}_j}+(\mathbf{S}_{-\mathbf{t}, \mathbf{x}_i})^{*}\mathbf{S}^{\rm{epi}(-\hat{h}^{-}_0)}_{-\mathbf{t},-\mathbf{x}_j}(u_i,u_j)
\end{equation*}
surrounded by projection $\bar{\chi}_{-\mathbf{a}}$.
It is nicer to have projection $\chi_{\mathbf{a}}$, so we change variables $u_i\mapsto-u_i$ and replace the Fredholm determinant of the kernel by that of its adjoint to get $\det\left(\mathbf{I}-\chi_{\mathbf{a}}\mathbf{K}^{{\rm hypo}(\hat{h}_0)}_{\mathbf{t}, {\rm ext}}\chi_{\mathbf{a}}\right)$ with $\mathbf{K}^{{\rm hypo}(\hat{h}_0)}_{\mathbf{t}, {\rm ext}}(u_i, u_j)=\mathbf{K}_{\lim}(\mathbf{x}_j, -u_j ; \mathbf{x}_i, -u_i)$.

By using $\textstyle(\mathbf{S}_{\mathbf{t}, \mathbf{x}})^{*}\mathbf{S}_{\mathbf{t}, \mathbf{-x}}=I$ and $\textstyle\mathbf{S}^{{\rm epi}(\hat{h})}_{-\mathbf{t},\mathbf{x}}(v,u)=\mathbf{S}^{{\rm hypo}(-\hat{h})}_{\mathbf{t},\mathbf{x}}(-v,-u)$ (see \cite{Quastel} for more information on these equations), we have 
\begin{equation*}
\mathbf{K}^{{\rm hypo}(\hat{h}_0)}_{\mathbf{t},{\rm ext}}(\mathbf{x}_i, \cdot ;\mathbf{x}_j, \cdot)=-e^{(\mathbf{x}_j-\mathbf{x}_i)\partial^2}\1_{\mathbf{x}_i<\mathbf{x}_j}+\left(\mathbf{S}^{{\rm hypo}(\hat{h}^{-}_0)}_{\mathbf{t},-\mathbf{x}_i}\right)^{*}\mathbf{S}_{\mathbf{t}, \mathbf{x}_j}.
\end{equation*}
\appendix
\section{The KPZ fixed point for the continuous time TASEP with jump rate $\beta$}
\label{Appendix}
In this section, we use our method to show that the KPZ fixed point is obtained in the continuous time TASEP  with jump rate $\beta\in(0, \infty)$.
From \eqref{beta232} and Proposition \ref{m13}, we get 
\begin{equation*}
\mathcal{M}(w)=e^{\beta(w-1)}.
\end{equation*}
Besides, by \eqref{gamgam}, we have
\begin{equation*}
\gamma(w)=e^{-\frac{\beta}{2}w}.
\end{equation*}
If  $\gamma(w)$ satisfies Assumption \ref{imas} and Assumption \ref{limcon}, we can show Theorem \ref{special} using Theorem \ref{Main} and Proposition \ref{scaling}. 
Therefore, we prove that $\gamma(w)$ satisfies Assumption \ref{imas} and Assumption \ref{limcon}.

First, we show that $\gamma(w)$ satisfies Assumption \ref{imas}.
Calculating the derivatives of $\gamma(w)$ up to the third order, we get
\begin{equation}
\label{appen22}
\gamma^{'}(w)=-\frac{\beta}{2}e^{-\frac{\beta}{2}w}, ~\gamma^{(2)}(w)=\frac{\beta^2}{4}e^{-\frac{\beta}{2}w}, ~\gamma^{(3)}(w)=-\frac{\beta^3}{8}e^{-\frac{\beta}{2}w}.
\end{equation}
Substituting $w=0$ for \eqref{appen22} gives
\begin{equation}
\label{appen23}
\gamma^{'}(0)=-\frac{\beta}{2}, ~\gamma^{(2)}(0)=\frac{\beta^2}{4}, ~\gamma^{(3)}(0)=-\frac{\beta^3}{8}.
\end{equation}
Thus, we obtain 
\begin{equation*}
\gamma^{(3)}(0)-3\gamma^{(2)}(0)\gamma^{'}(0)+2\{\gamma^{'}(0)\}^3-2\gamma^{'}(0)=\beta>0.
\end{equation*}
Next we prove that $\gamma(w)$ satisfies  Assumption \ref{limcon}.
By \eqref{D}, \eqref{E}, \eqref{F}, and \eqref{appen23}, we have
\begin{equation}
\label{appen234}
D=\frac{2}{\beta}, ~ E=\frac{1}{2}, ~ F=\frac{1}{2}.
\end{equation}
Note that $\log(1+x)< x$ for $x\in(-1, \infty)\setminus\{0\}$.
For $\theta\in [-\pi, -\frac{\pi}{3})\cup(\frac{\pi}{3}, \pi]$, by \eqref{appen234}, we obtain
\begin{equation*}
E\log|2-e^{i\theta}|+D\log|\gamma(1-e^{i\theta})|<\frac{1}{4}(4-4\cos\theta)+\cos\theta-1=0
\end{equation*}
and
\begin{equation*}
F\log|2-e^{i\theta}|-D\log|\gamma(e^{i\theta}-1)|<\frac{1}{4}(4-4\cos\theta)+\cos\theta-1=0.
\end{equation*}
This completes the proof.

\end{document}